\documentclass[reqno]{amsart}

\usepackage{amssymb,amsfonts,amsmath}

\newtheorem{theorem}{Theorem}
\newtheorem{lemma}{Lemma}
\newtheorem{corollary}{Corollary}
\newtheorem{proposition}{Proposition}
\newtheorem{remark}{Remark}

\DeclareMathOperator{\Aut}{Aut}
\DeclareMathOperator{\IA}{IA}
\DeclareMathOperator{\GL}{GL}

\newcommand{\fin}{\mathrm{fin}}

\title[Tame Automorphisms in Countably Infinite Rank]{Tame Automorphisms of Free Nilpotent Lie Algebras of Countably Infinite Rank}

\author{C. E. Kofinas}
\address{University of the Aegean, Department of Mathematics, Karlovassi GR-83200, Samos, Greece}
\email{kkofinas@aegean.gr}

\begin{document}

\begin{abstract}
Let $L_{\infty}$ be a free Lie algebra of countably infinite rank over a field of
characteristic $0$ and let $L_{\infty, c}$ be the free nilpotent Lie algebra of countably infinite rank and class $c$. We prove that every automorphism of $L_{\infty, c}$ is induced by an automorphism of $L_{\infty}$.

\medskip


\smallskip

\end{abstract}

\maketitle

\section{Introduction}\label{intro}

Let $K$ be a field of characteristic $0$ and let $L$ be a free Lie algebra over $K$. An 
ideal $V \subseteq L$ is called fully invariant if $\varphi(V) \subseteq V$ for every 
endomorphism $\varphi$ of $L$. In this case, each $\phi \in \Aut(L)$ induces an 
automorphism $\overline{\phi} \in \Aut(L/V)$ and hence yields a natural group 
homomorphism
\[
\widetilde{\rho}_{V} \colon \Aut(L) \longrightarrow \Aut(L/V), \quad \phi \mapsto \overline{\phi}.
\]
An automorphism of $L/V$ is called \emph{tame} if it lies in the image of 
$\widetilde{\rho}_{V}$; equivalently, if it lifts to an automorphism of $L$. Otherwise, 
it is called \emph{non-tame} (or \emph{wild}). This raises the following question:
\begin{quote}\itshape
Given a fully invariant ideal $V$ of $L$, is every automorphism of $L/V$ tame?
\end{quote}

In general, lifting problems for automorphisms of relatively free algebras may fail; 
general obstructions are described by Bryant--Drensky \cite{brdr}. Drensky 
\cite{dr92} constructed wild automorphisms for certain relatively free nilpotent-by-abelian 
Lie algebras of finite rank. Non-tame automorphisms also occur in free 
polynilpotent Lie algebras of finite rank; see Papistas \cite{pap93}. Further structural 
results on automorphism groups of relatively free algebras of countably infinite rank 
were obtained in \cite{brro}. More recently, Umirbaev \cite{umir2} obtained new 
results on tame and almost tame automorphisms of free metabelian Lie algebras of 
rank $n \geq 4$; see also \cite{umir1} for remarks on earlier proofs. 

The automorphism groups of relatively free nilpotent groups and Lie
algebras of finite rank may contain non-tame elements; see, for example, \cite{bglm} and 
\cite{drgu}. In contrast,
Bryant--Macedo\'nska \cite{brma} proved that every automorphism of a
relatively free nilpotent group of infinite rank is tame. For further background on tameness in the 
infinite-rank group case, see Roman'kov \cite{rom}. To the best of our knowledge, for 
free nilpotent Lie algebras of
countably infinite rank, the analogous tameness problem has remained
open.

We establish a Lie algebra analogue of the Bryant--Macedo\'nska result for countably 
infinite rank. In this setting, we construct controlled lifts and use them 
to obtain a global lift. Let $L_{\infty}$ denote the free Lie algebra of countably infinite 
rank over $K$, and let $L_{\infty, c} = L_{\infty}/\gamma_{c+1}(L_{\infty})$ be the 
free nilpotent Lie algebra of class $c$. We prove that every automorphism of 
$L_{\infty, c}$ is tame. As an application, we show that if $V$ is a fully invariant ideal 
of $L_{\infty}$ such that $L_{\infty}/V$ is nilpotent, then every automorphism of 
$L_{\infty}/V$ is tame.

\begin{theorem}\label{the4}
Let $L_{\infty, c}$ be the free nilpotent Lie algebra of countably infinite rank and 
class $c$. 
Then every automorphism of $L_{\infty, c}$ lifts to an automorphism of $L_{\infty}$.
\end{theorem}

\begin{corollary}\label{cor1}
Let $L_{\infty}$ be a free Lie algebra of countably infinite rank and let $V$ be a 
proper fully invariant ideal of $L_{\infty}$ such that $L_{\infty}/V$ is nilpotent. Then 
every automorphism of $L_{\infty}/V$ lifts to an automorphism of $L_{\infty}$.
\end{corollary}

\section{Preliminaries}\label{prel}

Let $\mathbb{N}$ denote the set of positive integers. Throughout,
$K$ is a field of characteristic $0$. All Lie algebras are taken over $K$. For a Lie 
algebra $M$ and $a, b \in M$, we write $[a, b]$ for the Lie commutator of $a$ and 
$b$. For $k \geq 3$ and $a_{1}, \dots, a_{k} \in M$, we use left-normed Lie
commutators, defined inductively by 
$[a_{1}, \dots, a_{k}] = [[a_{1}, \dots, a_{k-1}], a_{k}]$. For $i \in \mathbb{N}$, let $\gamma_{i}(M)$ denote the $i$-th 
term of the lower central series of $M$, and write $M^{\prime} = \gamma_{2}(M)$.

Let $L_{\infty}$ be the free Lie algebra of countably infinite rank freely generated by $\{\ell_{i} \colon i \in \mathbb{N}\}$. For $c \geq 1$, set $L_{\infty, c} = L_{\infty}/\gamma_{c+1}(L_{\infty})$ and write $y_{i} = \ell_{i} + \gamma_{c+1}(L_{\infty})$ for all $i \in \mathbb{N}$. Then $L_{\infty, c}$ is the free nilpotent Lie algebra of countably infinite rank and class $c$, freely generated by $\{y_{i} \colon i \in \mathbb{N}\}$. For $m \geq 1$, let $L_{\infty, c}^{m}$ denote the $K$-vector subspace of $L_{\infty, c}$ spanned by all Lie commutators of length $m$ in the elements $y_{i}$. In particular, $L_{\infty, c}^{1} = \langle y_{i} \colon i \in \mathbb{N} \rangle$. Since $L_{\infty, c}$ is free nilpotent over a field of characteristic $0$, it decomposes as $L_{\infty, c} = \bigoplus_{m=1}^{c} L_{\infty, c}^{m}$.

Let $A_{\infty} = L_{\infty, c}/L_{\infty, c}^{\prime}$ and write $\bar{y}_{i} = y_{i} + L_{\infty, c}^{\prime}$ for all $i \in \mathbb{N}$. Then the set $\{\bar{y}_{i} \colon i \in \mathbb{N}\}$ is a basis of the countably infinite-dimensional $K$-vector space $A_{\infty}$. 
Set
\[
\begin{aligned}
\GL_{\infty}(K) &= \Aut_{K}(A_{\infty}), \\
G_{\infty} &= \{\beta \in \GL_{\infty}(K) \colon \beta(\bar{y}_{i}) = \bar{y}_{i} \text{ for all but finitely many } i\}.
\end{aligned}
\]
Since $L_{\infty, c}$ is freely generated by $\{y_i \colon i \in \mathbb{N}\}$, every $\beta \in \GL_{\infty}(K)$ extends to an automorphism of $L_{\infty, c}$. Thus $\GL_{\infty}(K)$ embeds naturally into $\Aut(L_{\infty, c})$, and we identify both $\GL_{\infty}(K)$ and $G_{\infty}$ with their images. The group $\GL_{\infty}(K)$ acts on the $K$-vector space $A_{\infty}$, and this action extends diagonally to $L_{\infty, c}$ by
\[
g\bigl([y_{i_{1}}, \dots, y_{i_{k}}]\bigr) = [g(y_{i_{1}}), \dots, g(y_{i_{k}})]
\qquad (g \in \GL_{\infty}(K),\ i_{1}, \dots, i_{k} \in \mathbb{N}).
\]

\section{Automorphisms of Free Nilpotent Lie Algebras}\label{sec3}

\subsection{IA-Automorphisms}

The following results are well-known for free nilpotent Lie algebras.

\begin{lemma}\label{lemm1}
An endomorphism of $L_{\infty, c}$ is an automorphism if and only if it induces an automorphism of $L_{\infty, c}/L_{\infty, c}^{\prime}$.
\end{lemma}

\begin{lemma}\label{lemm2}
Every automorphism of $L_{\infty, c-1}$ lifts to an automorphism of $L_{\infty, c}$ for all $c \geq 2$.
\end{lemma}

For $2 \leq t \leq c$, the natural projection $\pi_{t} \colon L_{\infty, c} \twoheadrightarrow L_{\infty, c}/\gamma_{t}(L_{\infty, c})$ induces a group homomorphism $\widetilde{\pi}_{t} \colon \Aut(L_{\infty, c}) \longrightarrow \Aut(L_{\infty, c}/\gamma_{t}(L_{\infty, c}))$. By repeated application of Lemma \ref{lemm2}, it follows that each $\widetilde{\pi}_{t}$ is surjective. We write
$\IA(t,L_{\infty,c})=\ker\widetilde{\pi}_t$ for
$2\leq t\leq c$, and set
$\IA(c+1,L_{\infty,c})=\{\mathrm{Id}_{L_{\infty,c}}\}$. Thus, for
$t=2,\dots,c+1$,
\[
\IA(t, L_{\infty, c}) = \{\varphi \in \Aut(L_{\infty, c}) \colon \varphi(y_{i}) = y_{i} + v_{i}, v_{i} \in \gamma_{t}(L_{\infty, c}) \text{ for all } i \in \mathbb{N}\}. 
\]
For $t = 2$, write $\IA(L_{\infty, c}) = \IA(2, L_{\infty, c})$. The elements of $\IA(L_{\infty, c})$ are called IA-automorphisms of $L_{\infty, c}$. Since $\GL_{\infty}(K) \cap \IA(L_{\infty, c}) = \{\mathrm{Id}_{L_{\infty, c}}\}$, $\Aut(L_{\infty, c})$ is the semidirect product of $\IA(L_{\infty, c})$ by $\GL_{\infty}(K)$. Since $L_{\infty, c}$ is nilpotent, $\IA(L_{\infty, c})$ is also nilpotent and the following is a central series of $\IA(L_{\infty, c})$:
\[
\IA(L_{\infty, c}) > \IA(3, L_{\infty, c}) > \cdots > \IA(c, L_{\infty, c}) > \IA(c+1, L_{\infty, c}) = \{\mathrm{Id}_{L_{\infty, c}}\},
\]
where $\mathrm{Id}_{L_{\infty, c}}$ denotes the identity mapping. We denote 
\[
\overline{\IA}(t, L_{\infty, c}) = \IA(t, L_{\infty, c})/\IA(t+1, L_{\infty, c}).
\]
The analogous central series of IA-automorphisms for free nilpotent groups goes back to Andreadakis \cite{and}. For $2\leq t\leq c$, the group $\GL_{\infty}(K)$ acts on $\overline{\IA}(t, L_{\infty, c})$ by conjugation: $g \cdot \bigl(\phi\,\IA(t+1, L_{\infty, c})\bigr) = (g \phi g^{-1})\,\IA(t+1, L_{\infty, c})$, where $g \in \GL_{\infty}(K)$ and $\phi \in \IA(t, L_{\infty, c})$.

For $1 \leq t \leq c$, set 
\[
D_{t} = \gamma_{t}(L_{\infty, c})/\gamma_{t+1}(L_{\infty, c}). 
\]
In particular, $D_{1} = L_{\infty, c}/L_{\infty, c}^{\prime} = A_{\infty}$. For $2\leq t\leq c$, we consider the direct product of countably many copies of $D_{t}$:
\[
P(t) = \prod_{i\in\mathbb{N}} D_{t}.
\]

For $\phi\in\IA(t,L_{\infty,c})$, write
$\phi(y_i)=y_i+f_i$, where
$f_i\in\gamma_t(L_{\infty,c})$, and set
$\overline{f}_i=f_i+\gamma_{t+1}(L_{\infty,c})$. Since
$\IA(L_{\infty,c})$ acts trivially by conjugation on
$\overline{\IA}(t,L_{\infty,c})$, the mapping
\[
\phi\IA(t+1,L_{\infty,c})
\longmapsto
(\overline{f}_1,\overline{f}_2,\dots)
\]
defines an isomorphism of abelian groups
$\overline{\IA}(t,L_{\infty,c})\cong P(t)$. Indeed, for every
$(\overline{f}_i)_{i\in\mathbb N}\in P(t)$, choose representatives
$f_i\in\gamma_t(L_{\infty,c})$ and define
$\phi(y_i)=y_i+f_i$. By Lemma \ref{lemm1}, $\phi$ is an
automorphism. Moreover, composition corresponds to coordinatewise
addition modulo $\gamma_{t+1}(L_{\infty,c})$.

Let $g\in G_{\infty}$. Since $g$ preserves
$\gamma_t(L_{\infty,c})$ and $\gamma_{t+1}(L_{\infty,c})$, it
induces a $K$-linear automorphism of $D_t$. We denote this induced
map also by $g$. Write $g^{-1}=(b_{ij})$ with respect to the basis
$\{\bar y_i:i\in\mathbb N\}$ of $D_1$. For every fixed $j$, only
finitely many $b_{ij}$ are nonzero. Under the above identification $\overline{\IA}(t, L_{\infty, c}) \cong P(t)$, the conjugation action of $G_{\infty}$ is given by
\[
g \ast (\overline{f}_{1}, \overline{f}_{2}, \dots) = (g(\overline{f}_{1}), g(\overline{f}_{2}), \dots) g^{-1},
\]
that is,
\[
(g \ast \overline{f})_{j} = \sum_{i \in \mathbb{N}} g(\overline{f}_{i})\, b_{ij} \qquad (j \in \mathbb{N}).
\]
This is well-defined since for each $j$ the sum is finite. Hence, $P(t)$ is a $KG_{\infty}$-module. It is straightforward to verify that $\overline{\IA}(t, L_{\infty, c})$ is isomorphic to $P(t)$ as $KG_{\infty}$-modules. Thus, for $t = 2, \dots, c$, we identify $\overline{\IA}(t, L_{\infty, c})$ with $P(t)$ as $KG_{\infty}$-modules. This is the extension to countably infinite rank of the action described in \cite{drgu} for free nilpotent Lie algebras of finite rank.

\subsection{Finitary Automorphisms}

We adopt the corresponding terminology from \cite{brma}. Let $M$ denote
either $L_{\infty}$ or $L_{\infty,c}$, and let
$X=\{x_i:i\in\mathbb N\}$ be a free generating set of $M$. An
automorphism $\xi$ of $M$ is called \emph{finitary with respect to $X$}
if $\xi(x_i)=x_i$
for all but finitely many $i$. When $X$ is the standard free generating
set, we simply say that $\xi$ is finitary. Thus, the standard free
generating set is $\{\ell_i:i\in\mathbb N\}$ when $M=L_{\infty}$ and
$\{y_i:i\in\mathbb N\}$ when $M=L_{\infty,c}$.

For $2\leq t\leq c+1$, we define $\IA_{\fin}(t,L_{\infty,c})$ to be the set of
all finitary elements of $\IA(t,L_{\infty,c})$, with respect to the
standard free generating set $\{y_i:i\in\mathbb N\}$, and write
\[
\IA_{\fin}(L_{\infty,c})
=\IA_{\fin}(2,L_{\infty,c}).
\]
For $2\leq t\leq c$, set
\[
\overline{\IA}_{\fin}(t,L_{\infty,c})
=
\IA_{\fin}(t,L_{\infty,c})
/
\IA_{\fin}(t+1,L_{\infty,c}).
\]
Note that for $t = 2, \dots, c$,
\[
\IA_{\fin}(t, L_{\infty, c}) \cap \IA(t+1, L_{\infty, c}) = \IA_{\fin}(t+1, L_{\infty, c}).
\]
Hence, $\overline{\IA}_{\fin}(t, L_{\infty, c})$ is naturally isomorphic to a subgroup of $\overline{\IA}(t, L_{\infty, c})$.

Let $Q(t)$ denote the restricted direct sum of countably many copies of $D_{t}$. 
Thus, $Q(t)$ consists precisely of all tuples $u = (u_{1}, u_{2}, \dots)$ with 
$u_{i} \in D_{t}$ and $u_{i} = 0$ for all but finitely many $i$. 
From now on, the support of a tuple $(u_{1}, u_{2}, \dots)$ is defined to be the set 
of indices $i \in \mathbb{N}$ such that $u_{i} \neq 0$. The support of an 
automorphism is defined analogously as the set of generators on which it acts 
nontrivially.

Restricting the above isomorphism $\overline{\IA}(t, L_{\infty, c}) \cong P(t)$ to 
$\overline{\IA}_{\fin}(t, L_{\infty, c})$, we obtain 
$\overline{\IA}_{\fin}(t, L_{\infty, c}) \cong Q(t)$ as abelian groups. 
Moreover, every $g \in G_{\infty}$ fixes all but finitely many basis elements 
$\bar{y}_{i}$ of $D_{1}$. Hence, the matrix of $g^{-1}$ differs from the identity in 
only finitely many columns. It follows that for each $u \in Q(t)$, the tuple $g \ast u$ 
also has finite support. Therefore, $Q(t)$ is a $KG_{\infty}$-submodule of $P(t)$. 
Consequently, for each $t = 2, \dots, c$, we identify 
$\overline{\IA}_{\fin}(t, L_{\infty, c})$ with $Q(t)$ as $KG_{\infty}$-modules. In 
particular, since 
$\IA_{\fin}(c, L_{\infty, c}) = \overline{\IA}_{\fin}(c, L_{\infty, c})$, we identify 
$\IA_{\fin}(c, L_{\infty, c})$ with $Q(c)$ as $KG_{\infty}$-modules.

\begin{remark}\label{rema1}
\upshape
Under the identification $\overline{\IA}_{\fin}(t, L_{\infty, c}) \cong Q(t)$, the action 
of $G_{\infty}$ on $Q(t)$ agrees with conjugation in 
$\overline{\IA}_{\fin}(t, L_{\infty, c})$: if $u \in Q(t)$ 
corresponds to $\overline{\theta} \in \overline{\IA}_{\fin}(t, L_{\infty, c})$ and 
$g \in G_{\infty}$, then $g \ast u$ corresponds to $g \overline{\theta} g^{-1}$. In
particular, since $\overline{\IA}_{\fin}(t, L_{\infty, c})$ is an abelian group, the 
element $g \ast u - u$ corresponds to 
$(g \, \overline{\theta} \, g^{-1}) \, \overline{\theta}^{-1}$.
\end{remark}

\section{Lifting Automorphisms}\label{sec5}

For $c\geq2$, we call an automorphism $\theta\in\IA(c,L_{\infty,c})$
\emph{standard} if there exist pairwise distinct indices
$p,q_2,\dots,q_c$ and $a\in K$ such that
\[
\theta(y_p)=y_p+a[y_p,y_{q_2},\dots,y_{q_c}],
\qquad
\theta(y_i)=y_i\quad(i\neq p).
\]

\begin{proposition}\label{propo2}
Let $c \geq 2$. Then the standard automorphism $\sigma$ of $L_{\infty,c}$,
defined by
\[
\sigma(y_{1})=y_{1}+[y_{1},y_2,\dots,y_{c}], \qquad
\sigma(y_{i})=y_{i} \text{ for } i\neq1,
\]
lifts to an automorphism of $L_{\infty}$.
\end{proposition}

\begin{proof}
First assume that $c=2$. For each positive integer $n$, define the
automorphisms $\phi_n$, $\zeta_n$ and $\psi_n$ of $L_{\infty}$ by
\[
\begin{split}
\phi_n\colon \ell_{2n-1}
&\mapsto \ell_{2n-1}+\ell_{2n+1},
\qquad \ell_i\mapsto\ell_i \text{ for } i\neq2n-1,\\
\zeta_n\colon \ell_{2n+1}
&\mapsto \ell_{2n+1}-[\ell_2,\ell_{2n-1}],
\qquad \ell_i\mapsto\ell_i \text{ for } i\neq2n+1,\\
\psi_n\colon \ell_{2n-1}
&\mapsto \ell_{2n-1}-[\ell_2,\ell_{2n+1}],
\qquad \ell_i\mapsto\ell_i \text{ for } i\neq2n-1.
\end{split}
\]
Write
\[
\omega_n=\psi_n\zeta_n^{-1}\phi_n^{-1}\zeta_n\phi_n.
\]
A direct calculation modulo $\gamma_3(L_{\infty})$ shows that
$\omega_n$ induces the automorphism $\sigma_n$ of $L_{\infty,2}$
given by
\[
\begin{split}
\sigma_n(y_{2n-1})
&=y_{2n-1}-[y_2,y_{2n-1}],\\
\sigma_n(y_{2n+1})
&=y_{2n+1}+[y_2,y_{2n+1}],
\end{split}
\]
and
\[
\sigma_n(y_i)=y_i
\qquad(i\neq2n-1,2n+1).
\]

Define the maps $\omega_{\mathrm{o}}$ and
$\omega_{\mathrm{o}}^{-}$ of $L_{\infty}$ on the free generators by
\[
\omega_{\mathrm{o}}(\ell_i)=
\begin{cases}
\omega_n(\ell_i),
& i\in \{2n-1,2n+1\}\text{ for some odd }n,\\
\ell_i,
& \text{otherwise},
\end{cases}
\]
and
\[
\omega_{\mathrm{o}}^{-}(\ell_i)=
\begin{cases}
\omega_n^{-1}(\ell_i),
& i\in \{2n-1,2n+1\}\text{ for some odd }n,\\
\ell_i,
& \text{otherwise}.
\end{cases}
\]
Since the pairs $\{2n-1,2n+1\}$ corresponding to odd values of $n$
are pairwise disjoint, the above maps are well-defined and determine
endomorphisms of $L_{\infty}$. Every $\omega_n$ and $\omega_n^{-1}$
fixes $\ell_2$ and preserves the Lie subalgebra generated by
$\ell_2,\ell_{2n-1},\ell_{2n+1}$. It follows that, on every free
generator,
\[
\omega_{\mathrm{o}}^{-}\omega_{\mathrm{o}}
=\omega_{\mathrm{o}}\omega_{\mathrm{o}}^{-}
=\mathrm{Id}_{L_{\infty}}.
\]
Thus $\omega_{\mathrm{o}}$ is an automorphism.

The pairs $\{2n-1,2n+1\}$ corresponding to even values of $n$ are
also pairwise disjoint. Therefore, by the same argument as above, the
endomorphism $\omega_{\mathrm{e}}$ of $L_{\infty}$ defined by
\[
\omega_{\mathrm{e}}(\ell_i)=
\begin{cases}
\omega_n(\ell_i),
& i\in \{2n-1,2n+1\}\text{ for some even }n,\\
\ell_i,
& \text{otherwise},
\end{cases}
\]
is an automorphism.

Write $\omega=\omega_{\mathrm{o}}\omega_{\mathrm{e}}$. The
automorphisms induced by $\omega_{\mathrm{o}}$ and
$\omega_{\mathrm{e}}$ belong to $\IA(2,L_{\infty,2})$ and therefore
fix $\gamma_2(L_{\infty,2})$ pointwise. For $m\geq0$, the automorphism
induced by $\omega_{\mathrm{o}}$ satisfies
\[
\begin{split}
y_{4m+1}
&\mapsto y_{4m+1}-[y_2,y_{4m+1}],\\
y_{4m+3}
&\mapsto y_{4m+3}+[y_2,y_{4m+3}],
\end{split}
\]
and fixes every even generator. For $m\geq1$, the automorphism induced
by $\omega_{\mathrm{e}}$ satisfies
\[
\begin{split}
y_{4m-1}
&\mapsto y_{4m-1}-[y_2,y_{4m-1}],\\
y_{4m+1}
&\mapsto y_{4m+1}+[y_2,y_{4m+1}],
\end{split}
\]
and fixes every even generator and $y_1$.

It follows that
\[
\omega(\ell_1)+\gamma_3(L_{\infty})
=y_1-[y_2,y_1]
=y_1+[y_1,y_2].
\]
For $m\geq1$,
\[
\omega(\ell_{4m+1})+\gamma_3(L_{\infty})
=y_{4m+1}-[y_2,y_{4m+1}]
+[y_2,y_{4m+1}]
=y_{4m+1},
\]
whereas, for $m\geq0$,
\[
\omega(\ell_{4m+3})+\gamma_3(L_{\infty})
=y_{4m+3}+[y_2,y_{4m+3}]
-[y_2,y_{4m+3}]
=y_{4m+3}.
\]
Every even generator is fixed. Therefore $\omega$ induces $\sigma$ on
$L_{\infty,2}$.

Now assume that $c\geq3$. Define the automorphisms $\beta_1$ and
$\xi_1$ of $L_{\infty}$ by
\[
\begin{split}
\beta_1\colon \ell_1
&\mapsto \ell_1+[\ell_{c+1},\ell_3,\dots,\ell_c],
\qquad \ell_i\mapsto\ell_i \text{ for } i\neq1,\\
\xi_1\colon \ell_{c+1}
&\mapsto \ell_{c+1}+[\ell_1,\ell_2],
\qquad \ell_i\mapsto\ell_i \text{ for } i\neq c+1.
\end{split}
\]
Let $\tau=(2\ 3\ \dots\ c)\in S_c$. For $n\geq2$, define the
automorphisms $\beta_n$ and $\xi_n$ of $L_{\infty}$ by
\[
\begin{split}
\beta_n\colon \ell_{n+c-1}
&\mapsto \ell_{n+c-1}
+[\ell_{n+c},\ell_{\tau^{n-1}(3)},\dots,
\ell_{\tau^{n-1}(c)}],\\
\xi_n\colon \ell_{n+c}
&\mapsto \ell_{n+c}
+[\ell_{n+c-1},\ell_{\tau^{n-1}(2)}],
\end{split}
\]
and let them fix all the remaining free generators.

Write
\[
\eta_n=\xi_n^{-1}\beta_n^{-1}\xi_n\beta_n
\qquad(n\geq1).
\]
A direct calculation modulo $\gamma_{c+1}(L_{\infty})$ shows that
$\eta_1$ induces the automorphism $\sigma_1$ given by
\[
\begin{split}
\sigma_1(y_1)
&=y_1+[y_1,y_2,\dots,y_c],\\
\sigma_1(y_{c+1})
&=y_{c+1}-[y_{c+1},y_3,\dots,y_c,y_2],
\end{split}
\]
and fixes all the other generators. For $n\geq2$, the automorphism
$\eta_n$ induces the automorphism $\sigma_n$ given by
\[
\begin{split}
\sigma_n(y_{n+c-1})
&=y_{n+c-1}
+[y_{n+c-1},y_{\tau^{n-1}(2)},\dots,
y_{\tau^{n-1}(c)}],\\
\sigma_n(y_{n+c})
&=y_{n+c}
-[y_{n+c},y_{\tau^{n-1}(3)},\dots,
y_{\tau^{n-1}(c)},y_{\tau^{n-1}(2)}],
\end{split}
\]
and fixes all the other generators.

Write
\[
\Psi_1=\{1,c+1\},
\qquad
\Psi_n=\{n+c-1,n+c\}
\quad(n\geq2).
\]
Every $\eta_n$ fixes the free generators whose indices do not belong
to $\Psi_n$. Moreover, every $\eta_n$ fixes
$\ell_2,\dots,\ell_c$.

The pairs $\Psi_n$ corresponding to odd values of $n$ are pairwise
disjoint. Define the maps $\eta_{\mathrm{o}}$ and
$\eta_{\mathrm{o}}^{-}$ of $L_{\infty}$ on the free generators by
\[
\eta_{\mathrm{o}}(\ell_i)=
\begin{cases}
\eta_n(\ell_i),
& i\in\Psi_n\text{ for some odd }n,\\
\ell_i,
& \text{otherwise},
\end{cases}
\]
and
\[
\eta_{\mathrm{o}}^{-}(\ell_i)=
\begin{cases}
\eta_n^{-1}(\ell_i),
& i\in\Psi_n\text{ for some odd }n,\\
\ell_i,
& \text{otherwise}.
\end{cases}
\]
The above maps are well-defined and determine endomorphisms of
$L_{\infty}$. Every $\eta_n$ and $\eta_n^{-1}$ preserves the Lie
subalgebra generated by the generators indexed by $\Psi_n$ together
with $\ell_2,\dots,\ell_c$. It follows that, on every free generator,
\[
\eta_{\mathrm{o}}^{-}\eta_{\mathrm{o}}
=\eta_{\mathrm{o}}\eta_{\mathrm{o}}^{-}
=\mathrm{Id}_{L_{\infty}}.
\]
Thus $\eta_{\mathrm{o}}$ is an automorphism.

The pairs $\Psi_n$ corresponding to even values of $n$ are also
pairwise disjoint. Therefore, by the same argument as above, the
endomorphism $\eta_{\mathrm{e}}$ of $L_{\infty}$ defined by
\[
\eta_{\mathrm{e}}(\ell_i)=
\begin{cases}
\eta_n(\ell_i),
& i\in\Psi_n\text{ for some even }n,\\
\ell_i,
& \text{otherwise},
\end{cases}
\]
is an automorphism.

Write $\eta=\eta_{\mathrm{o}}\eta_{\mathrm{e}}$. The automorphisms
induced by $\eta_{\mathrm{o}}$ and $\eta_{\mathrm{e}}$ belong to
$\IA(c,L_{\infty,c})$ and therefore fix
$\gamma_c(L_{\infty,c})$ pointwise. The generator $y_1$ is changed
only by $\sigma_1$. Hence
\[
\eta(\ell_1)+\gamma_{c+1}(L_{\infty})
=y_1+[y_1,y_2,\dots,y_c].
\]
The generators $y_2,\dots,y_c$ are fixed.

Let $n\geq1$. The generator $y_{n+c}$ is changed only by
$\sigma_n$ and $\sigma_{n+1}$. In particular,
\[
\begin{split}
\sigma_n(y_{n+c})&=y_{n+c}
-[y_{n+c},y_{\tau^{n-1}(3)},\dots,
y_{\tau^{n-1}(c)},y_{\tau^{n-1}(2)}],\\
\sigma_{n+1}(y_{n+c})&=y_{n+c}
+[y_{n+c},y_{\tau^n(2)},\dots,y_{\tau^n(c)}].
\end{split}
\]
Since
\[
(\tau^n(2),\dots,\tau^n(c))
=
(\tau^{n-1}(3),\dots,\tau^{n-1}(c),
\tau^{n-1}(2)),
\]
the two Lie commutators above are equal. One of $\sigma_n$ and
$\sigma_{n+1}$ occurs in the automorphism induced by
$\eta_{\mathrm{o}}$, and the other occurs in the automorphism induced
by $\eta_{\mathrm{e}}$. Since both induced automorphisms fix
$\gamma_c(L_{\infty,c})$ pointwise, we obtain
\[
\begin{split}
\eta(\ell_{n+c})+\gamma_{c+1}(L_{\infty})
&=y_{n+c}
-[y_{n+c},y_{\tau^{n-1}(3)},\dots,
y_{\tau^{n-1}(c)},y_{\tau^{n-1}(2)}]\\
&\quad
+[y_{n+c},y_{\tau^n(2)},\dots,y_{\tau^n(c)}]\\
&=y_{n+c}.
\end{split}
\]
Therefore $\eta$ induces $\sigma$ on $L_{\infty,c}$.
\end{proof}

The next result is a relative version of the module calculation in
\cite[proof of Proposition~4.1]{drgu}.
We give the details because the
requirement that the generators indexed by $B$ remain fixed is
essential.

\begin{lemma}\label{lem:controlled-top}
Let $c\geq2$. Let $B,\Delta,\Omega\subseteq\mathbb N$ be finite sets
such that
$B\cap\Delta=\varnothing$ and $\Delta\subseteq\Omega$, and let
$R\subseteq\mathbb N\setminus(B\cup\Omega)$ be infinite. Denote by
$L_{\Omega,c}$ the Lie subalgebra of $L_{\infty,c}$ generated by
$\{y_i:i\in\Omega\}$. Let
$\alpha\in\IA_{\fin}(c,L_{\infty,c})$ satisfy the conditions:
\begin{enumerate}
\item[(i)] Its support is contained in $\Delta$;

\item[(ii)] $\alpha(y_i)-y_i\in\gamma_c(L_{\Omega,c})$
for all $i\in\Delta$.
\end{enumerate}

Then there exist a finite set $R_0\subseteq R$, finitary linear
automorphisms $g_1,\dots,g_M$ of $L_{\infty,c}$, and standard 
automorphisms $\theta_1,\dots,\theta_M$ such that
\[
\alpha=\prod_{\nu=1}^{M}g_\nu\theta_\nu g_\nu^{-1}.
\]
They may be chosen so that:
\begin{enumerate}
\item every $g_\nu$ fixes $y_j$ for $j\in B$, preserves
$\langle y_j:j\in\Omega\cup R_0\rangle$, and fixes every $y_j$ with
$j\notin\Omega\cup R_0$;

\item every $\theta_\nu$ has the form
\[
y_{s_\nu}\mapsto
y_{s_\nu}+a_\nu
[y_{s_\nu},y_{q_{\nu,2}},\dots,y_{q_{\nu,c}}],
\]
where $a_\nu\in K$, $s_\nu\in\Delta\cup R_0$,
$q_{\nu,2},\dots,q_{\nu,c}\in R_0$, and these indices are pairwise
distinct.
\end{enumerate}
In particular, every $\theta_\nu$ and every
$g_\nu\theta_\nu g_\nu^{-1}$ fixes the generators indexed by $B$.
\end{lemma}

\begin{proof}
Using the identification
$\IA_{\fin}(c,L_{\infty,c})\cong Q(c)$ established above, write
$e_i(u)$ for the tuple whose $i$-th coordinate is
$u\in\gamma_c(L_{\infty,c})$ and whose remaining coordinates are
zero. By (i), (ii), and linearity, it is enough to consider
$e_p([y_{i_1},\dots,y_{i_c}])$, where
$p\in\Delta$ and $i_1,\dots,i_c\in\Omega$. Fix such a tuple.

Let
$G_{\infty}(B)=\{g\in G_{\infty}:g(y_j)=y_j
\text{ for every }j\in B\}$.
In what follows, we use only finitary elements of
$G_{\infty}(B)$. Choose pairwise distinct indices
$r_2,\dots,r_c\in R$, and let $N$ be the
$KG_{\infty}(B)$-submodule of $Q(c)$ generated by
$e_p([y_p,y_{r_2},\dots,y_{r_c}])$. This tuple corresponds to a
standard automorphism. All auxiliary indices introduced
below are chosen in $R$, pairwise distinct and different from the
indices already used.

\smallskip
\noindent\emph{{\bf Step 1}.}
By repeated use of the Jacobi identity and anticommutativity, every Lie
commutator under consideration in which $y_p$ occurs exactly once is
a linear combination of left-normed Lie commutators
$[y_p,y_{j_2},\dots,y_{j_c}]$, where
$j_2,\dots,j_c\in\Omega\setminus\{p\}$. It is therefore enough to
consider such a Lie commutator. Write
$w_1=[y_p,y_{r_2},\dots,y_{r_c}]$ and, for $2\leq k\leq c$, write
\[
w_k=[y_p,y_{j_2},\dots,y_{j_k},
y_{r_{k+1}},\dots,y_{r_c}].
\]
Let $2\leq k\leq c$ and suppose that
$e_p(w_{k-1})\in N$. Let $h_k$ be defined by
$h_k(y_{r_k})=y_{r_k}+y_{j_k}$ and let it fix all the remaining
generators. Since $r_k\notin B$, we have
$h_k\in G_{\infty}(B)$. Moreover, $y_{r_k}$ occurs exactly once in
$w_{k-1}$. Since $h_k$ fixes $y_p$ and $j_k\neq p$, we obtain
\[
h_k\ast e_p(w_{k-1})-e_p(w_{k-1})=e_p(w_k).
\]
Since $e_p(w_1)\in N$, induction gives
$e_p([y_p,y_{j_2},\dots,y_{j_c}])=e_p(w_c)\in N$.

\smallskip
\noindent\emph{{\bf Step 2}.}
Let $s\in\Delta\cup R$, and let $v$ be a Lie commutator of length $c$
in which $y_s$ occurs exactly once. Choose pairwise distinct
auxiliary indices $q_2,\dots,q_c\in R$, different from $s$, from
$r_2,\dots,r_c$, and from all the indices occurring in $v$. Write
\[
S=\{p,s,r_2,\dots,r_c,q_2,\dots,q_c\}.
\]
The partial map which sends $p$ to $s$ and $r_k$ to $q_k$ for
$k=2,\dots,c$ is injective, and therefore extends to a permutation
$\pi$ of the finite set $S$. Extend $\pi$ by fixing every index
outside $S$. Then $\pi$ is finitary and supported on $S$. Since
$S\cap B=\varnothing$, we have $\pi\in G_{\infty}(B)$. Moreover,
\[
\pi\ast e_p([y_p,y_{r_2},\dots,y_{r_c}])
=e_s([y_s,y_{q_2},\dots,y_{q_c}])\in N.
\]
Repeating the argument of Step 1 with $s$ in place of $p$ and
$q_k$ in place of $r_k$, we obtain $e_s(v)\in N$.

\smallskip
\noindent\emph{{\bf Step 3}.}
Suppose that none of $i_1,\dots,i_c$ is equal to $p$. Write
$u=[y_p,y_{i_2},\dots,y_{i_c}]$. By Step 1, $e_p(u)\in N$.
Let $g$ be defined by $g(y_p)=y_p+y_{i_1}$ and let it fix all the
remaining generators. Since $p\notin B$, we have
$g\in G_{\infty}(B)$. Moreover,
\[
g(u)=u+[y_{i_1},y_{i_2},\dots,y_{i_c}],
\]
and 
\[
g\ast e_p(u)-e_p(u)
=e_p([y_{i_1},y_{i_2},\dots,y_{i_c}]).
\]
Hence $e_p([y_{i_1},\dots,y_{i_c}])\in N$. The same argument, with
$s$ in place of $p$ and using Step 2, shows that $e_s(v)\in N$
whenever $s\in\Delta\cup R$ and $v$ is a Lie commutator of length $c$
which does not contain $y_s$.

\smallskip
\noindent\emph{{\bf Step 4}.}
Let $w$ be a Lie commutator of length $c$ which contains $y_p$ and is
linear in every generator different from $y_p$. We use induction on
the number of occurrences of $y_p$. If $y_p$ occurs exactly once, 
the assertion follows from Step 2,
applied with $s=p$.

Suppose that $y_p$ occurs $d\geq2$ times. Choose $r\in R$ which does
not occur in $w$, replace one occurrence of $y_p$ by $y_r$, and
denote the resulting Lie commutator by $v$. Then $y_p$ occurs $d-1$
times in $v$, and $v$ is linear in every generator different from
$y_p$. Hence, by induction, $e_p(v)\in N$.

Let $g$ be defined by $g(y_r)=y_r+y_p$ and let it fix all the
remaining generators. Since $r\notin B$, we have
$g\in G_{\infty}(B)$. Moreover, $g^{-1}(y_r)=y_r-y_p$. Since $y_r$
occurs exactly once in $v$, we have $g(v)=v+w$, and
\[
\begin{split}
g\ast e_p(v)
&=e_p(g(v))-e_r(g(v))\\
&=e_p(v)+e_p(w)-e_r(v)-e_r(w).
\end{split}
\]
By Step 2, $e_r(v)\in N$, and by Step 3, $e_r(w)\in N$. Since
$g\ast e_p(v)\in N$ and $e_p(v)\in N$, it follows that
$e_p(w)\in N$.

\smallskip
\noindent\emph{{\bf Step 5}.}
Let $w=[y_{i_1},\dots,y_{i_c}]$ be an arbitrary Lie commutator in which
$y_p$ occurs at least twice. Replace every occurrence of a generator
different from $y_p$ by a different auxiliary generator indexed by
$R$, and denote the resulting Lie commutator by $v$. Then $v$ is linear
in every generator different from $y_p$, and hence Step 4 gives
$e_p(v)\in N$.

Let $y_{s_1},\dots,y_{s_t}$ be the auxiliary generators occurring
in $v$, and let $y_{j_k}$ be the generator of $w$ corresponding to
$y_{s_k}$. Write $v_0=v$, and, for $1\leq k\leq t$, let $v_k$ be
obtained from $v_{k-1}$ by replacing $y_{s_k}$ by $y_{j_k}$. Thus
$v_t=w$.

For $1\leq k\leq t$, let $h_k$ be defined by
$h_k(y_{s_k})=y_{s_k}+y_{j_k}$ and let it fix all the remaining
generators. Since $s_k\notin B$, we have
$h_k\in G_{\infty}(B)$. Moreover, $y_{s_k}$ occurs exactly once in
$v_{k-1}$, and hence multilinearity gives
\[
h_k(v_{k-1})=v_{k-1}+v_k.
\]
Since $s_k\in R$ and $p\in\Delta\subseteq\Omega$, we have
$s_k\neq p$. Moreover, $j_k\neq p$, since the auxiliary generator
$y_{s_k}$ replaced an occurrence of a generator different from
$y_p$. Thus $h_k$ fixes $y_p$, and $h_k^{-1}(y_i)$ has no
$y_p$-component for $i\neq p$. Hence 
\[
h_k\ast e_p(v_{k-1})
=e_p(h_k(v_{k-1}))
=e_p(v_{k-1})+e_p(v_k).
\]
Therefore
\[
h_k\ast e_p(v_{k-1})-e_p(v_{k-1})=e_p(v_k).
\]
Starting with $e_p(v_0)=e_p(v)\in N$, induction gives
$e_p(v_k)\in N$ for every $k$. Consequently,
$e_p(w)=e_p(v_t)\in N$.

Steps 3, 1, and 5 cover, respectively, the cases in which $y_p$
occurs zero times, exactly once, and at least twice. Hence
$e_p([y_{i_1},\dots,y_{i_c}])\in N$. 
Tracing the constructions in Steps 1--5, the fixed tuple
can be written as a finite sum
\[
e_p([y_{i_1},\dots,y_{i_c}])
=
\sum_{\nu=1}^{M}a_\nu g_\nu\ast u_\nu,
\]
where $a_\nu\in K$ and every
$g_\nu\in G_{\infty}(B)$ is finitary. Moreover,
\[
u_\nu
=e_{s_\nu}
([y_{s_\nu},y_{q_{\nu,2}},\dots,y_{q_{\nu,c}}]),
\]
where $s_\nu,q_{\nu,2},\dots,q_{\nu,c}$ are distinct,
$s_\nu\in\Delta\cup R$, and
$q_{\nu,2},\dots,q_{\nu,c}\in R$.
Each $g_\nu$ is a finite product of the transvections and
permutations used above.

Write
\[
\alpha=\sum_{p\in\Delta}e_p(v_p),
\qquad
v_p=\sum_{\mu}b_{p,\mu}w_{p,\mu},
\]
where each sum is finite, $b_{p,\mu}\in K$, and every
$w_{p,\mu}$ is a Lie commutator of length $c$ involving only generators
indexed by $\Omega$. Applying
the preceding construction to each $e_p(w_{p,\mu})$, multiplying the
resulting expression by $b_{p,\mu}$, and combining and renumbering the finitely many sums, we obtain
\[
\alpha=\sum_{\nu=1}^{M}a_\nu g_\nu\ast u_\nu.
\]

For each $\nu$, let $\theta_\nu$ be the standard 
automorphism corresponding, under the identification
$\IA_{\fin}(c,L_{\infty,c})\cong Q(c)$, to the tuple
$a_\nu u_\nu$. Thus, if
\[
u_\nu
=
e_{s_\nu}
([y_{s_\nu},y_{q_{\nu,2}},\dots,y_{q_{\nu,c}}]),
\]
then
\[
\theta_\nu(y_{s_\nu})
=
y_{s_\nu}
+
a_\nu
[y_{s_\nu},y_{q_{\nu,2}},\dots,y_{q_{\nu,c}}],
\]
and $\theta_\nu$ fixes all the remaining free generators.

By Remark \ref{rema1}, the tuple
$g_\nu\ast(a_\nu u_\nu)
=
a_\nu g_\nu\ast u_\nu
$
corresponds to the automorphism
$g_\nu\theta_\nu g_\nu^{-1}$.
Since addition in $Q(c)$ corresponds to multiplication in
$\IA_{\fin}(c,L_{\infty,c})$, it follows that
\[
\alpha
=
\prod_{\nu=1}^{M}
g_\nu\theta_\nu g_\nu^{-1}.
\]

Let $R_0\subseteq R$ be the finite set of all indices from $R$
occurring in the preceding constructions. Every transvection and
permutation used above fixes the generators indexed by $B$, maps
$\langle y_j:j\in\Omega\cup R_0\rangle$ onto itself, and fixes every
generator whose index does not belong to $\Omega\cup R_0$. Since
every $g_\nu$ is a finite product of these transformations, this
proves (1). By construction, every $u_\nu$ has the form
$e_{s_\nu}([y_{s_\nu},y_{q_{\nu,2}},\dots,y_{q_{\nu,c}}])$, where
$s_\nu,q_{\nu,2},\dots,q_{\nu,c}$ are distinct,
$s_\nu\in\Delta\cup R_0$, and
$q_{\nu,2},\dots,q_{\nu,c}\in R_0$. This proves (2). Since
$(\Delta\cup R_0)\cap B=\varnothing$, every $\theta_\nu$ fixes the
generators indexed by $B$. Since every $g_\nu$ also fixes these
generators, so does every
$g_\nu\theta_\nu g_\nu^{-1}$.
\end{proof}

We now combine Proposition \ref{propo2} and Lemma
\ref{lem:controlled-top} to obtain the controlled lifting result
needed below.

\begin{proposition}\label{prop:controlled-lifting}
Let $c\geq2$. Let $X=\{x_i:i\in\mathbb N\}$ be a free generating
set of $L_{\infty}$, and write
$\overline{x}_i=x_i+\gamma_{c+1}(L_{\infty})$. Let
$\Gamma,\Delta\subseteq\mathbb N$ satisfy
$\Gamma\cap\Delta=\varnothing$, $|\Delta|<\infty$, and
$|\mathbb N\setminus(\Gamma\cup\Delta)|=\infty$. Suppose that
$\alpha\in\IA(c,L_{\infty,c})$ and
$\alpha(\overline{x}_j)=\overline{x}_j$ for every $j\in\Gamma$.

Then there exists an automorphism $\xi$ of $L_{\infty}$ fixing
$x_j$ for every $j\in\Gamma$ such that the automorphism induced by
$\xi$ on $L_{\infty,c}$ belongs to
$\IA(c,L_{\infty,c})$ and
\[
\xi(x_i)+\gamma_{c+1}(L_{\infty})=
\begin{cases}
\alpha(\overline{x}_i),&i\in\Delta,\\
\overline{x}_i,&i\notin\Delta.
\end{cases}
\]
\end{proposition}

\begin{proof}
We first prove the result for the standard free generating set
$X=\{\ell_i:i\in\mathbb N\}$. Write
$\alpha(y_i)=y_i+v_i$, where
$v_i\in\gamma_c(L_{\infty,c})$. Let $\alpha_{\Delta}$ be the
endomorphism of $L_{\infty,c}$ defined by
\[
\alpha_{\Delta}(y_i)=y_i+v_i\quad(i\in\Delta),
\qquad
\alpha_{\Delta}(y_j)=y_j\quad(j\notin\Delta).
\]
Since $c\geq2$, $\gamma_c(L_{\infty,c})\subseteq
\gamma_2(L_{\infty,c})=L_{\infty,c}'$. Therefore
$\alpha_{\Delta}$ induces the identity on
$L_{\infty,c}/L_{\infty,c}'$ and hence, by
Lemma \ref{lemm1}, it is an automorphism. Moreover,
its support is contained in the finite set $\Delta$ and
$\alpha_{\Delta}(y_i)-y_i\in\gamma_c(L_{\infty,c})$ for every
$i\in\mathbb N$. Therefore
$\alpha_{\Delta}\in\IA_{\fin}(c,L_{\infty,c})$.

Choose a finite set $\Omega\subseteq\mathbb N$ which contains
$\Delta$ and all the indices occurring in fixed commutator
expressions for the elements $v_i$, where $i\in\Delta$. Write
$B=\Gamma\cap\Omega$. Then $B\cap\Delta=\varnothing$. Since
$\mathbb N\setminus(\Gamma\cup\Delta)$ is infinite and $\Omega$ is
finite, the set $\mathbb N\setminus(\Gamma\cup\Omega)$ is infinite.
Choose an infinite set $R\subseteq\mathbb N\setminus(\Gamma\cup\Omega)$.

The hypotheses of Lemma \ref{lem:controlled-top} are now satisfied.
Indeed, $B,\Delta,\Omega$ are finite,
$B\cap\Delta=\varnothing$, $\Delta\subseteq\Omega$, and
$R\subseteq\mathbb N\setminus(\Gamma\cup\Omega)
\subseteq\mathbb N\setminus(B\cup\Omega)$.
Moreover, $\alpha_{\Delta}\in\IA_{\fin}(c,L_{\infty,c})$, its
support is contained in $\Delta$, and
$v_i\in\gamma_c(L_{\Omega,c})$ for every $i\in\Delta$.

Therefore, by Lemma \ref{lem:controlled-top}, there exist a finite
set $R_0\subseteq R$, finitary linear automorphisms
$g_1,\dots,g_s$, and standard automorphisms
$\theta_1,\dots,\theta_s$ such that
\[
\alpha_{\Delta}
=\rho_1\cdots\rho_s,
\qquad
\rho_k=g_k\theta_k g_k^{-1}.
\]
For every $k$, the automorphism $g_k$ fixes $y_j$ for $j\in B$,
preserves
$\langle y_j:j\in\Omega\cup R_0\rangle$, and fixes every $y_j$ with
$j\notin\Omega\cup R_0$. Moreover, $\theta_k$ has the form
\[
\theta_k(y_{p_k})
=y_{p_k}
+a_k[y_{p_k},y_{q_{k,2}},\dots,y_{q_{k,c}}],
\qquad
\theta_k(y_i)=y_i\quad(i\neq p_k),
\]
where $a_k\in K$, the indices
$p_k,q_{k,2},\dots,q_{k,c}$ are pairwise distinct,
$p_k\in\Delta\cup R_0$, and
$q_{k,2},\dots,q_{k,c}\in R_0$.

Every $g_k$ fixes the generators indexed by $\Gamma$. Indeed, if
$j\in\Gamma\cap\Omega=B$, then $g_k(y_j)=y_j$ by Lemma
\ref{lem:controlled-top}. If $j\in\Gamma\setminus\Omega$, then
$j\notin R_0$, since $R_0\subseteq R$ and
$R\cap\Gamma=\varnothing$. Hence
$j\notin\Omega\cup R_0$, and Lemma \ref{lem:controlled-top} again
gives $g_k(y_j)=y_j$.
Since $(\Delta\cup R_0)\cap\Gamma=\varnothing$, choose a countably
infinite set
$I\subseteq\mathbb N\setminus\Gamma$ 
containing $\Delta\cup R_0$, and let $L(I)$ be the free Lie subalgebra
of $L_{\infty}$ generated by $\{\ell_i:i\in I\}$.
For each $k=1,\dots,s$, after relabelling the free generators of
$L(I)$, Proposition
\ref{propo2} gives an automorphism $\widehat{\sigma}_k$ of $L(I)$
which induces the coefficient-one standard automorphism
with nonzero coordinate $p_k$ and commutator indices
$q_{k,2},\dots,q_{k,c}$.

If $a_k\neq0$, let $d_k$ be the linear automorphism of $L(I)$ defined
by
\[
d_k(\ell_{q_{k,2}})=a_k\ell_{q_{k,2}},
\qquad
d_k(\ell_i)=\ell_i\quad(i\in I,\ i\neq q_{k,2}),
\]
and write
$\widehat{\theta}_k
=d_k\widehat{\sigma}_k d_k^{-1}$.
Then $\widehat{\theta}_k$ induces $\theta_k$ on
$L(I)/\gamma_{c+1}(L(I))$. If $a_k=0$, let
$\widehat{\theta}_k$ be the identity automorphism of $L(I)$.
Extend $\widehat{\theta}_k$ to an automorphism of $L_{\infty}$ by
fixing every $\ell_i$ with $i\notin I$. Since
$\gamma_{c+1}(L(I))\subseteq\gamma_{c+1}(L_{\infty})$, the extended
automorphism induces $\theta_k$ on $L_{\infty,c}$. Since
$I\cap\Gamma=\varnothing$, it fixes every $\ell_j$ with
$j\in\Gamma$.

Let $\widetilde{g}_k$ be the linear automorphism of $L_{\infty}$
defined on the free generators by the same formulas as $g_k$. It
induces $g_k$ on $L_{\infty,c}$ and fixes every $\ell_j$ with
$j\in\Gamma$. Define
\[
\xi=
\prod_{k=1}^{s}
\widetilde{g}_k\widehat{\theta}_k\widetilde{g}_k^{-1}.
\]
Every factor in this product fixes the generators indexed by
$\Gamma$. Hence
$\xi(\ell_j)=\ell_j$ for every 
$j\in\Gamma$.
Furthermore,
\[
\overline{\xi}
=\prod_{k=1}^{s}g_k\theta_k g_k^{-1}
=\alpha_{\Delta}.
\]
Therefore $\overline{\xi}\in\IA(c,L_{\infty,c})$,
$\overline{\xi}(y_i)=\alpha(y_i)$ for every $i\in\Delta$, and
$\overline{\xi}(y_j)=y_j$ for every $j\notin\Delta$. This proves the
result for the standard free generating set.

Now let $X=\{x_i:i\in\mathbb N\}$ be an arbitrary free generating set
of $L_{\infty}$. Choose $\lambda\in\Aut(L_{\infty})$ such that
$\lambda(\ell_i)=x_i$ for every $i\in\mathbb N$, and let
$\overline{\lambda}$ be the automorphism induced by $\lambda$ on
$L_{\infty,c}$. Since $\IA(c,L_{\infty,c})$ is normal in
$\Aut(L_{\infty,c})$, we have
\[
\alpha_0
=\overline{\lambda}^{-1}\alpha\overline{\lambda}
\in\IA(c,L_{\infty,c}).
\]
For every $j\in\Gamma$,
\[
\alpha_0(y_j)=\overline{\lambda}^{-1}\alpha\overline{\lambda}(y_j)=\overline{\lambda}^{-1}\alpha(\overline{x}_j)=\overline{\lambda}^{-1}(\overline{x}_j)=y_j.
\]

Apply the case already proved for the standard free generating set
to $\alpha_0$, with the same index sets $\Gamma$ and $\Delta$, and
let $\xi_0$ be the resulting automorphism of $L_{\infty}$. If
$\overline{\xi_0}$ denotes the automorphism induced by $\xi_0$ on
$L_{\infty,c}$, then $\xi_0(\ell_j)=\ell_j$ for every
$j\in\Gamma$, $\overline{\xi_0}\in\IA(c,L_{\infty,c})$, and
\[
\overline{\xi_0}(y_i)=
\begin{cases}
\alpha_0(y_i),&i\in\Delta,\\
y_i,&i\notin\Delta.
\end{cases}
\]

Write $\xi=\lambda\xi_0\lambda^{-1}$. Then, for every
$j\in\Gamma$,
\[
\xi(x_j)=\lambda\xi_0(\ell_j)=\lambda(\ell_j)=x_j.
\]
Moreover,
$\overline{\xi}
=\overline{\lambda}\,\overline{\xi_0}\,\overline{\lambda}^{-1}$,
so $\overline{\xi}\in\IA(c,L_{\infty,c})$ by normality. Since
$\overline{\lambda}(y_i)=\overline{x}_i$ and
$\overline{\lambda}\alpha_0=\alpha\overline{\lambda}$, we obtain
\[
\overline{\xi}(\overline{x}_i)
=
\begin{cases}
\alpha(\overline{x}_i),&i\in\Delta,\\
\overline{x}_i,&i\notin\Delta.
\end{cases}
\]
\end{proof}

A finite subset of a free generating set of $L_{\infty}$ will be
called a \emph{finite primitive system}. Equivalently, it is a finite
subset of $L_{\infty}$ which is contained in some free generating set.

The proof of the following proposition adapts the
construction used in the proof of
\cite[Proposition~2]{brma} to the present Lie-algebraic setting.

\begin{proposition}\label{propo3}
Let $c\geq2$. Then every automorphism in
$\IA(c,L_{\infty,c})$ lifts to an automorphism of $L_{\infty}$.
\end{proposition}

\begin{proof}
Let $\pi\colon L_{\infty}\longrightarrow L_{\infty,c}$ be the natural
epimorphism. Write $H=\IA(c,L_{\infty,c})$ and fix $\alpha\in H$. We construct finite primitive systems
\[
U_0\subseteq U_1\subseteq U_2\subseteq\cdots,
\qquad
V_0\subseteq V_1\subseteq V_2\subseteq\cdots,
\]
and automorphisms $\theta_n\in\Aut(L_{\infty})$ with the following
properties. If $h_n$ denotes the automorphism induced by $\theta_n$ on
$L_{\infty,c}$, then
\begin{align}
h_n&\in H, \label{eq:globalization-top-IA}\\
V_n&=\theta_n(U_n), \label{eq:globalization-images}\\
\pi(\theta_n(u))&=\alpha(\pi(u))
\qquad(u\in U_n). \label{eq:globalization-compatibility}
\end{align}
Moreover, for every $n\geq1$, the restriction of
$\theta_n\colon U_n\to V_n$ to $U_{n-1}$ agrees with
$\theta_{n-1}\colon U_{n-1}\to V_{n-1}$.

Start with $U_0=V_0=\varnothing$ and
$\theta_0=\mathrm{Id}_{L_{\infty}}$. The automorphism induced by
$\theta_0$ is the identity, which belongs to $H$, and the remaining
conditions are immediate.

\smallskip
\noindent\emph{{\bf Step 1}.}
Suppose that $U_{2m}$, $V_{2m}$ and $\theta_{2m}$ have been
constructed. Choose and enumerate a free generating set
$X=\{x_i:i\in\mathbb N\}$ containing $U_{2m}$. Since $X$ generates $L_{\infty}$, the standard free generator
$\ell_{m+1}$ belongs to the Lie subalgebra generated by a finite subset
of $X$. Hence there is a finite subset
$E\subseteq X\setminus U_{2m}$ such that
\[
\ell_{m+1}\in L(U_{2m}\cup E).
\]
Since $U_{2m}$ and $E$ are disjoint finite subsets of $X$, there
exist unique disjoint finite sets
$\Gamma_X,\Delta_X\subseteq\mathbb N$ such that
\[
U_{2m}=\{x_i:i\in\Gamma_X\},
\qquad
E=\{x_i:i\in\Delta_X\}.
\]  
In particular,
$|\mathbb N\setminus(\Gamma_X\cup\Delta_X)|=\infty$.

Write $\beta=h_{2m}^{-1}\alpha$. Since $h_{2m},\alpha\in H$, we have
$\beta\in H$. We claim that $\beta$ fixes $\pi(x_i)$ for every
$i\in\Gamma_X$. Indeed, for $i\in\Gamma_X$, we have $x_i\in U_{2m}$, and hence
\eqref{eq:globalization-compatibility} gives
\[
\alpha(\pi(x_i))
=\pi(\theta_{2m}(x_i))
=h_{2m}(\pi(x_i)).
\]
Consequently,
\[
\beta(\pi(x_i))=h_{2m}^{-1}\alpha(\pi(x_i))=\pi(x_i).
\]

The hypotheses of Proposition \ref{prop:controlled-lifting} are
satisfied: $X$ is a free generating set,
$\Gamma_X$ and $\Delta_X$ are finite and disjoint,
$\mathbb N\setminus(\Gamma_X\cup\Delta_X)$ is infinite,
$\beta\in H=\IA(c,L_{\infty,c})$, and $\beta$ fixes
$\pi(x_i)$ for every $i\in\Gamma_X$. Therefore Proposition \ref{prop:controlled-lifting} gives
an automorphism $\xi$ of $L_{\infty}$ which fixes $U_{2m}$ pointwise
and whose induced automorphism $\kappa$ belongs to $H$. Moreover,
\[
\pi(\xi(e))=\beta(\pi(e))
\qquad(e\in E).
\]

Define 
\[
\theta_{2m+1}=\theta_{2m}\xi,\quad
U_{2m+1}=U_{2m}\cup E, \quad
V_{2m+1}=\theta_{2m+1}(U_{2m+1}).
\]
Since
$U_{2m+1}\subseteq X$, the set $U_{2m+1}$ is a finite primitive
system. Moreover, since $\theta_{2m+1}$ is an automorphism and $X$ is a free
generating set, $\theta_{2m+1}(X)$ is a free generating set. Since
$V_{2m+1}\subseteq\theta_{2m+1}(X)$, the set $V_{2m+1}$ is also a
finite primitive system.

The automorphism induced by $\theta_{2m+1}$ is
$h_{2m+1}=h_{2m}\kappa$, and therefore $h_{2m+1}\in H$. If
$u\in U_{2m}$, then $\xi(u)=u$, and hence
\[
\theta_{2m+1}(u)=\theta_{2m}\xi(u)=\theta_{2m}(u).
\] 
Thus $\theta_{2m+1}$ agrees with $\theta_{2m}$ on $U_{2m}$. In particular,
\[
V_{2m}=\theta_{2m+1}(U_{2m})\subseteq V_{2m+1}.
\]
For $e\in E$, we have
\[
\pi(\theta_{2m+1}(e))
=\pi(\theta_{2m}\xi(e))
=h_{2m}(\pi(\xi(e)))
=h_{2m}\beta(\pi(e))
=\alpha(\pi(e)).
\]
Together with the preceding agreement on $U_{2m}$ and
\eqref{eq:globalization-compatibility} for $n=2m$, this proves
\eqref{eq:globalization-compatibility} on $U_{2m+1}$. Thus
\eqref{eq:globalization-top-IA}--\eqref{eq:globalization-compatibility}
are preserved. By the choice of $E$,
\begin{equation}\label{eq:globalization-forth-generation}
\ell_{m+1}\in L(U_{2m+1}).
\end{equation}

\smallskip
\noindent\emph{{\bf Step 2}.}
Suppose that $U_{2m+1}$, $V_{2m+1}$ and $\theta_{2m+1}$ have been
constructed. Choose and enumerate a free generating set
$Y=\{z_i:i\in\mathbb N\}$ containing $V_{2m+1}$. Since $Y$ generates
$L_{\infty}$, the standard free generator $\ell_{m+1}$ belongs to
the Lie subalgebra generated by a finite subset of $Y$. Hence there
is a finite subset $D\subseteq Y\setminus V_{2m+1}$ such that
\[
\ell_{m+1}\in L(V_{2m+1}\cup D).
\]
Since $V_{2m+1}$ and $D$ are disjoint finite subsets of $Y$, there
exist unique disjoint finite sets
$\Gamma_Y,\Delta_Y\subseteq\mathbb N$ such that
\[
V_{2m+1}=\{z_i:i\in\Gamma_Y\},
\qquad
D=\{z_i:i\in\Delta_Y\}.
\]
In particular,
$|\mathbb N\setminus(\Gamma_Y\cup\Delta_Y)|=\infty$.

Write $\gamma=\alpha h_{2m+1}^{-1}$. Since
$\alpha,h_{2m+1}\in H$, we have $\gamma\in H$. We claim that
$\gamma$ fixes $\pi(z_i)$ for every $i\in\Gamma_Y$. Indeed, for $i\in\Gamma_Y$, we have
$z_i\in V_{2m+1}=\theta_{2m+1}(U_{2m+1})$. Hence there exists
$u\in U_{2m+1}$ such that $z_i=\theta_{2m+1}(u)$.
Hence
$\pi(z_i)=h_{2m+1}(\pi(u))$,
and, by \eqref{eq:globalization-compatibility},
\[
\gamma(\pi(z_i))=\alpha h_{2m+1}^{-1}(h_{2m+1}(\pi(u)))=\alpha(\pi(u))=\pi(\theta_{2m+1}(u))=\pi(z_i).
\]
It follows that $\gamma^{-1}\in H$ also fixes $\pi(z_i)$ for every
$i\in\Gamma_Y$.

The hypotheses of Proposition \ref{prop:controlled-lifting} are
satisfied: $Y$ is a free generating set,
$\Gamma_Y$ and $\Delta_Y$ are finite and disjoint,
$\mathbb N\setminus(\Gamma_Y\cup\Delta_Y)$ is infinite,
$\gamma^{-1}\in H=\IA(c,L_{\infty,c})$, and $\gamma^{-1}$ fixes
$\pi(z_i)$ for every $i\in\Gamma_Y$. Therefore Proposition
\ref{prop:controlled-lifting} gives an automorphism $\zeta$ of
$L_{\infty}$ which fixes $V_{2m+1}$ pointwise and whose induced
automorphism $\mu$ belongs to $H$. Moreover,
\[
\pi(\zeta(d))=\gamma^{-1}(\pi(d))
\qquad(d\in D).
\]

Define
\[
\theta_{2m+2}=\zeta^{-1}\theta_{2m+1},\qquad U_{2m+2}=U_{2m+1}\cup\theta_{2m+1}^{-1}(\zeta(D)),\qquad V_{2m+2}=V_{2m+1}\cup D.
\]
The two sets occurring in the union defining $U_{2m+2}$ are
disjoint. Indeed, suppose that $\theta_{2m+1}^{-1}(\zeta(d))\in U_{2m+1}$
for some $d\in D$. Then
$\zeta(d)\in\theta_{2m+1}(U_{2m+1})=V_{2m+1}$.
Since $\zeta^{-1}$ fixes $V_{2m+1}$ pointwise, it follows that
$d=\zeta^{-1}(\zeta(d))\in V_{2m+1}$,
contrary to $D\cap V_{2m+1}=\varnothing$.
Since $V_{2m+2}=V_{2m+1}\cup D\subseteq Y$,
the set $V_{2m+2}$ is a finite primitive system. Moreover, since
$\zeta$ fixes $V_{2m+1}$ pointwise,
\[
U_{2m+2}=\theta_{2m+1}^{-1}(V_{2m+1}\cup\zeta(D))=\theta_{2m+1}^{-1}\zeta(V_{2m+1}\cup D)\subseteq\theta_{2m+1}^{-1}\zeta(Y).
\]
Since $\theta_{2m+1}^{-1}\zeta$ is an automorphism of
$L_{\infty}$ and $Y$ is a free generating set, the set
$\theta_{2m+1}^{-1}\zeta(Y)$ is also a free generating set of
$L_{\infty}$. Hence $U_{2m+2}$ is also a finite primitive system.

The automorphism induced by $\theta_{2m+2}$ is
$h_{2m+2}=\mu^{-1}h_{2m+1}$.
Since $\mu,h_{2m+1}\in H$, we have $h_{2m+2}\in H$.

Since $\zeta^{-1}$ fixes $V_{2m+1}$ pointwise,
\[
\theta_{2m+2}(U_{2m+1})=\zeta^{-1}\theta_{2m+1}(U_{2m+1})=\zeta^{-1}(V_{2m+1})=V_{2m+1}.
\]
Furthermore, for every $d\in D$,
\[
\theta_{2m+2}\bigl(\theta_{2m+1}^{-1}(\zeta(d))\bigr)=\zeta^{-1}\theta_{2m+1}\theta_{2m+1}^{-1}(\zeta(d))=d.
\]
Therefore
\[
V_{2m+2}=\theta_{2m+2}(U_{2m+2}).
\]

If $u\in U_{2m+1}$, then
$\theta_{2m+1}(u)\in V_{2m+1}$, and hence
\[
\theta_{2m+2}(u)=\zeta^{-1}\theta_{2m+1}(u)=\theta_{2m+1}(u).
\]
Thus $\theta_{2m+2}$ agrees with $\theta_{2m+1}$ on
$U_{2m+1}$.

It remains to verify
\eqref{eq:globalization-compatibility} on the newly added elements.
Let $e\in U_{2m+2}\setminus U_{2m+1}$.
By the definition of $U_{2m+2}$ and the disjointness proved above,
there exists $d\in D$ such that $e=\theta_{2m+1}^{-1}(\zeta(d))$.
Hence
\[
\theta_{2m+1}(e)=\zeta(d)
\]
and
\[
\theta_{2m+2}(e)=\zeta^{-1}\theta_{2m+1}(e)=d.
\]
Since $h_{2m+1}$ is induced by $\theta_{2m+1}$, we have
\[
h_{2m+1}(\pi(e))=\pi(\theta_{2m+1}(e))=\pi(\zeta(d))
\]
and hence
\[
\pi(e)=h_{2m+1}^{-1}(\pi(\zeta(d))).
\]
On the other hand, the choice of $\zeta$ gives
\[
\pi(\zeta(d))=\gamma^{-1}(\pi(d)).
\]
Therefore
\[
\alpha(\pi(e))=\alpha h_{2m+1}^{-1}(\pi(\zeta(d)))=\gamma(\pi(\zeta(d)))=\gamma\gamma^{-1}(\pi(d))=\pi(d)=\pi(\theta_{2m+2}(e)).
\]
Together with the preceding agreement on $U_{2m+1}$ and
\eqref{eq:globalization-compatibility} for $n=2m+1$, this proves \eqref{eq:globalization-compatibility} on
$U_{2m+2}$. Hence
\eqref{eq:globalization-top-IA}--\eqref{eq:globalization-compatibility}
are again preserved.

Finally, by the choice of $D$ and the definition of $V_{2m+2}$,
\begin{equation}\label{eq:globalization-back-generation}
\ell_{m+1}\in L(V_{2m+1}\cup D)=L(V_{2m+2}).
\end{equation}

\smallskip
\noindent\emph{{\bf Step 3}.}
Continue the construction for all $m\geq0$, and write
\[
U=\bigcup_{n\geq0}U_n,
\qquad
V=\bigcup_{n\geq0}V_n.
\]
For every $i\geq1$, taking $m=i-1$ in
\eqref{eq:globalization-forth-generation} gives
$\ell_i\in L(U_{2i-1})\subseteq L(U)$, whereas
\eqref{eq:globalization-back-generation} gives
$\ell_i\in L(V_{2i})\subseteq L(V)$. Since the elements
$\ell_i$, $i\in\mathbb N$, generate $L_{\infty}$, both $U$ and $V$
generate $L_{\infty}$.

Any Lie relation among elements of $U$ involves only finitely many
elements. Since the sets $U_n$ form an increasing sequence, these
elements belong to some $U_n$. But $U_n$ is contained in a free
generating set of $L_{\infty}$, so no nonzero Lie relation can hold
among its elements. Thus $U$ is a free set. Since $U$ also generates
$L_{\infty}$, it is a free generating set of $L_{\infty}$. The same
argument shows that $V$ is a free generating set.

For $u\in U$, choose $n$ such that $u\in U_n$ and define
\[
\theta(u)=\theta_n(u).
\]
This definition is independent of the choice of $n$. Indeed, if
$u\in U_r\cap U_n$ and $r<n$, then $U_r\subseteq U_n$ and the
construction gives $\theta_n(u)=\theta_r(u)$. Moreover,
$\theta(u)\in V_n\subseteq V$, so this defines a map
$\theta\colon U\longrightarrow V$.

The map $\theta$ is injective. Indeed, if
$\theta(u)=\theta(u')$, we may choose $n$ such that $u,u'\in U_n$. Then
$\theta_n(u)=\theta_n(u')$, and, since $\theta_n$ is an
automorphism of $L_{\infty}$, we obtain $u=u'$. The map is also
surjective. If $v\in V$, then $v\in V_n$ for some $n$. Since
$V_n=\theta_n(U_n)$, there exists $u\in U_n$ such that
$v=\theta_n(u)=\theta(u)$.

Thus $\theta\colon U\longrightarrow V$ is a bijection. Since $U$
and $V$ are free generating sets of $L_{\infty}$, this bijection
extends uniquely to an automorphism, also denoted by
$\theta\in\Aut(L_{\infty})$.

Let $u\in U$ and choose $n$ such that $u\in U_n$. By the definition
of $\theta$ and \eqref{eq:globalization-compatibility},
\[
\pi(\theta(u))
=\pi(\theta_n(u))
=\alpha(\pi(u)).
\]
Therefore the homomorphisms
$\pi\theta,\alpha\pi\colon L_{\infty}\longrightarrow L_{\infty,c}$
agree on the free generating set $U$, and hence $\pi\theta=\alpha\pi$.
Thus $\theta$ lifts $\alpha$.
\end{proof}

\section{Proof of the Results}\label{sec6}
\subsection{Proof of the Main Theorem}
\begin{proof}[Proof of Theorem \ref{the4}]
We argue by induction on $c$. Suppose first that $c=1$, and let
$\alpha\in\Aut(L_{\infty,1})$. Since $L_{\infty,1}$ is the vector
space with basis $y_1,y_2,\dots$, the formulas defining $\alpha$ on
this basis define an endomorphism $\phi$ of $L_{\infty}$ when the
elements $y_i$ are replaced by $\ell_i$. Applying the same
construction to $\alpha^{-1}$ gives the inverse of $\phi$. Thus
$\phi\in\Aut(L_{\infty})$ and induces $\alpha$.

Assume now that $c \geq 2$ and that the statement holds for $c - 1$. Let $\alpha \in \Aut(L_{\infty, c})$. Recall that the natural projection $\pi_{c} \colon L_{\infty, c} \twoheadrightarrow L_{\infty, c}/\gamma_{c}(L_{\infty, c})$ induces a group homomorphism $\widetilde{\pi}_{c} \colon \Aut(L_{\infty, c}) \longrightarrow \Aut(L_{\infty, c}/\gamma_{c}(L_{\infty, c}))$. Since $L_{\infty, c}/\gamma_{c}(L_{\infty, c}) \cong L_{\infty, c-1}$, by the inductive hypothesis there exists $\phi \in \Aut(L_{\infty})$ inducing $\widetilde{\pi}_{c}(\alpha)$. Let $\alpha^{\prime} \in \Aut(L_{\infty, c})$ be the automorphism induced by $\phi$. Then $\widetilde{\pi}_{c}(\alpha) = \widetilde{\pi}_{c}(\alpha^{\prime})$ and hence $\delta = \alpha \circ (\alpha^{\prime})^{-1} \in \ker \widetilde{\pi}_{c} = \IA(c, L_{\infty, c})$. By Proposition \ref{propo3}, $\delta$ lifts to an automorphism $\widetilde{\delta} \in \Aut(L_{\infty})$. Therefore $\widetilde{\alpha} = \widetilde{\delta} \circ \phi$ is an automorphism of $L_{\infty}$ inducing $\alpha$ on $L_{\infty, c}$. This completes the induction.
\end{proof}

\subsection{Relatively Free Nilpotent Quotients}

For background on varieties of Lie algebras, see \cite{as} or
\cite{bah2}. Fully invariant ideals of a free Lie algebra are precisely
the verbal ideals. Hence, if $V$ is a fully invariant ideal of
$L_{\infty}$, then $L_{\infty}/V$ is relatively free in the
corresponding variety.

Let $V$ be a proper fully invariant ideal of $L_{\infty}$ such that
$L_{\infty}/V$ is nilpotent of class $c$. Then
$\gamma_{c+1}(L_{\infty})\subseteq V$. We first observe that $V$ is
homogeneous. For $\lambda\in K$, let $\delta_\lambda$ be the
endomorphism of $L_{\infty}$ defined by
$\delta_\lambda(\ell_i)=\lambda\ell_i$ for every $i\in\mathbb N$.
If $v=v_1+\cdots+v_m\in V$, where $v_d$ is homogeneous of degree
$d$, then
\[
\delta_\lambda(v)
=\lambda v_1+\lambda^2v_2+\cdots+\lambda^m v_m\in V
\]
for every $\lambda\in K$. Since $K$ is infinite, a Vandermonde
argument gives $v_d\in V$ for every $d$.
Furthermore, $V$ contains no nonzero element of degree $1$. Indeed,
if $0\neq v_1\in V$ were homogeneous of degree $1$, then for every
$j\in\mathbb N$, a suitable endomorphism of $L_{\infty}$ maps $v_1$
to $\ell_j$. Full invariance would therefore give $\ell_j\in V$ for
every $j$, and hence $V=L_{\infty}$, contrary to
$L_{\infty}/V\neq\{0\}$. Thus
$\gamma_{c+1}(L_{\infty})
\subseteq V
\subseteq L_{\infty}'$.

Set $M_{\infty,c}=L_{\infty}/V$ and write $z_i=\ell_i+V$. Let
$\rho_c\colon L_{\infty,c}\twoheadrightarrow M_{\infty,c}$ be defined
by $\rho_c(y_i)=z_i$. Its kernel is
$W=V/\gamma_{c+1}(L_{\infty})\subseteq L_{\infty,c}'$.
Moreover, $W$ is fully invariant in $L_{\infty,c}$. Indeed, every
endomorphism of $L_{\infty,c}$ is induced by an endomorphism of
$L_{\infty}$, and $V$ is fully invariant in $L_{\infty}$. It follows
that $\rho_c$ induces a group homomorphism
$\Aut(L_{\infty,c})\longrightarrow\Aut(M_{\infty,c})$.
Since $W\subseteq L_{\infty,c}'$, it also induces an isomorphism
$\overline{\rho}_c\colon
L_{\infty,c}/L_{\infty,c}'
\longrightarrow
M_{\infty,c}/M_{\infty,c}'$.

\begin{lemma}\label{lemm4}
The homomorphism
$\Aut(L_{\infty,c})\longrightarrow\Aut(M_{\infty,c})$
induced by $\rho_c$ is surjective.
\end{lemma}

\begin{proof}
Let $\alpha\in\Aut(M_{\infty,c})$, and let $\overline{\alpha}$ denote
the automorphism induced by $\alpha$ on
$M_{\infty,c}/M_{\infty,c}'$. For each $i\in\mathbb N$, choose
$w_i\in L_{\infty,c}$ such that $\rho_c(w_i)=\alpha(z_i)$, and define
an endomorphism $\Phi$ of $L_{\infty,c}$ by $\Phi(y_i)=w_i$. Then
$\rho_c\Phi=\alpha\rho_c$. Consequently, the endomorphism induced by
$\Phi$ on $L_{\infty,c}/L_{\infty,c}'$ is
$\overline{\Phi}
=\overline{\rho}_c^{-1}\,
\overline{\alpha}\,
\overline{\rho}_c$,
and is therefore an automorphism. By Lemma \ref{lemm1},
$\Phi\in\Aut(L_{\infty,c})$. Since
$\rho_c\Phi=\alpha\rho_c$, the image of $\Phi$ under the induced
homomorphism is $\alpha$.
\end{proof}

\begin{proof}[Proof of Corollary \ref{cor1}]
Let $\alpha\in\Aut(M_{\infty,c})$. By Lemma \ref{lemm4}, choose
$\beta\in\Aut(L_{\infty,c})$ inducing $\alpha$. By Theorem
\ref{the4}, the automorphism $\beta$ is induced by an automorphism
$\widetilde{\beta}$ of $L_{\infty}$. Since $V$ is fully invariant,
$\widetilde{\beta}$ preserves $V$ and induces an automorphism of
$L_{\infty}/V=M_{\infty,c}$. Since
$\rho_c\beta=\alpha\rho_c$, this induced automorphism is $\alpha$.
\end{proof}

\end{document}